\documentclass{amsart}
\font\ehsc=cmcsc10 scaled 850

\let\<=\langle
\let\>=\rangle
\let\\=\backslash
\let\sse=\subseteq
\let\noi=\noindent

\let\limply=\Longrightarrow

\def\0{\{0\}}

\def\span{{\kern.5pt{\rm span}\kern1pt}}
\def\smallfrac#1#2{{\textstyle{\frac{#1}{#2}}}}
\def\conv{{\;\longrightarrow\;}}
\def\wconv{{{\buildrel_{\scriptstyle w}\over\conv}}}
\def\sconv{{{\buildrel_{\scriptstyle s}\over\conv}}}
\def\uconv{{{\buildrel_{\scriptstyle u}\over\conv}}}

	\font\fiverm=cmr5
\def\sslash{\hbox{{\fiverm/}}}
\def\notconv{{{\conv\kern-13pt\slash}\kern9pt}}
\def\notuconv{{{\uconv\kern-13pt\sslash}\kern9pt}}
\def\notsconv{{{\sconv\kern-13pt\sslash}\kern9pt}}
\def\notwconv{{{\wconv\kern-13pt\sslash}\kern9pt}}

\def\newmatrix#1{\null\,\vcenter{
                   \baselineskip=8pt\mathsurround=-0pt\ialign{
                   \hfil ${##}$
                   \hfil &&
                   \hfil ${##}$
                   \hfil \crcr
                   \mathstrut \crcr
                   \noalign{\kern-\baselineskip}#1 \crcr
                   \mathstrut \crcr
                   \noalign{\kern-\baselineskip} \crcr }}\!}

\def\B{{\mathcal B}}
\def\C{{\kern.5pt\mathcal C}}
\def\H{{\mathcal H}}

\def\O{{\mathcal O}}

\def\U{{\mathcal U}}
\def\X{{\mathcal X}}

\def\BH{{\B[\H]}}

\def\CC{{\mathbb C\kern.5pt}}
\def\NN{{\mathbb N\kern.5pt}}
\newsymbol\varnothing 203F

\newtheorem{theorem}{Theorem}
\newtheorem{lemma}{Lemma}

\theoremstyle{definition}

\newtheorem{remark}{Remark}

\numberwithin{theorem}{section}
\numberwithin{lemma}{section}
\numberwithin{corollary}{section}
\numberwithin{proposition}{section}
\numberwithin{conjecture}{section}
\numberwithin{definition}{section}
\numberwithin{remark}{section}
\numberwithin{question}{section}

\begin{document}

\vglue-45pt\noindent
\hfill{\it Journal of Function Spaces}\/,
{\bf 2018} (article id 4732836 - 2018) 1--5

\vglue20pt
\title{Boundedly Spaced Subsequences and Weak Dynamics}
\author{C.S. Kubrusly}
\address{Applied Mathematics Department, Federal University,
         Rio de Janeiro, RJ, Brazil}
\email{carloskubrusly@gmail.com}
\author{P.C.M. Vieira}
\address{National Laboratory for Scientific Computation,
         Petr\'opolis, RJ, Brazil}
\email{paulocm@lncc.br}
\renewcommand{\keywordsname}{Keywords}
\date{March 7, 2018}

\begin{abstract}
Weak supercyclicity is related to weak stability, which leads to the question
that asks whether every weakly supercyclic power bounded operator is weakly
stable$.$ This is approached here by investigating weak l-sequential
supercyclicity for Hilbert-space contractions through Nagy--Folia\c s--Langer
decomposition, thus reducing the problem to the quest of conditions for a
weakly l-sequentially supercyclic unitary operator to be weakly stable, and
this is done in light of boundedly spaced subsequences.
\end{abstract}

\maketitle

%%%%%%%%%%%%%%%%%%%%%%%%%%%%%%%%%%%%%%%%%%%%%%%%%%%%%%%%%% SECTION 1
\vskip0pt\noi
\section{Introduction}

The purpose of this paper is to characterize weak supercyclicity for
Hilbert-space contractions, which is shown to be equivalent to characterizing
weak supercyclicity for unitary operators$.$ This is naturally motivated by
an open question that asks whether every weakly supercyclic power bounded
operator is weakly stable (which in turn is naturally motivated by a result
that asserts that every supercyclic power bounded operator is strongly
stable)$.$ Precisely, weakly supercyclicity is investigated in light of
boundedly spaced subsequences as discussed in Lemma 3.1$.$ The main result in
Theorem 4.1 characterizes weakly l-sequentially supercyclic unitary
opera\-tors $U\!$ that are weakly unstable in terms of boundedly spaced
subsequences of the power sequence $\{U^n\}.$ Remark 4.1 shows that
characterizing any form of weak super\-cyclicity for weakly unstable unitary
operators is equivalent to characterizing any form of weak supercyclicity for
weakly unstable contractions after the Nagy--Foia\c s--Langer decomposition.

%%%%%%%%%%%%%%%%%%%%%%%%%%%%%%%%%%%%%%%%%%%%%%%%%%%%%%%%%% SECTION 2
\vskip0pt\noi
\section{Notation and Terminology}

Throughout this paper $\H$ denotes a complex (infinite-dimensional)
Hilbert space, and $\BH$ denotes the Banach algebra of all operators on
$\H$ (i.e., of all bounded linear transformations of $\H$ into itself)$.$
An operator ${T\in\BH}$ is power bounded if ${\sup_n\|T^n\|<\infty}$ (same
notation for norm in $\H$ and for the induced uniform norm in $\BH).$ An
operator ${U\!\in\BH}$ is unitary if $UU^*\!=U^*U\!=I$, where $I$ stands for
the identity in $\BH$ and ${T^*\kern-1pt\in\BH}$ denotes the adjoint of an
operator ${T\in\BH}.$ An operator ${T\in\BH}$ is strongly stable or weakly
stable if the $\H$-valued power sequence $\{T^nx\}_{n\ge0}\kern-1pt$ converges
strongly (i.e., in the norm topology) or weakly to zero for every ${x\in\H}.$
In other words, if
$$
T^nx\conv0
\quad\;\hbox{or}\;\quad
T^nx\wconv0
$$
for every ${x\in\H}$, which means ${\|T^nx\|\to0}$ for every ${x\in\H}$
or ${\<T^nx\,;y\>\to0}$ for every ${x,y\in\H}$, respectively (clearly, strong
stability implies weak stability)$.$ Let
$$
\O_T(y)={\bigcup}_{n\ge0}T^ny=\big\{T^ny\in\H\!:\,n\ge0\big\}
$$
denote the orbit of a vector ${y\in\H}$ under an operator ${T\in\BH}$ --- we
write ${\bigcup}_{n\ge0}T^ny$ for
${\bigcup}_{n\ge0}T^n(\{y\})={\bigcup}_{n\ge0}\{T^ny\}.$ The orbit $\O_T(A)$
of a set ${A\sse\H}$ under $T$ is likewise defined:
$\O_T(A)=\bigcup_{n\ge0}T^n(A).$ In particular, the orbit of the
one-dimensional space spanned by $y$,
$$
\O_T(\span\{y\})={\bigcup}_{n\ge0}T^n(\span\{y\})
=\big\{\alpha T^ny\in\H\!:\,\alpha\in\CC,\,n\ge0\big\},
$$
is referred to as the projective orbit of a vector ${y\in\H}$ under an
operator ${T\in\BH}.$ A non\-zero vector ${y\in\H}$ is a {\it supercyclic
vector}\/ for an operator ${T\in\BH}$ if the projective orbit of $y$ is
dense in $\H$ in the norm topology; that is, if
$$
\O_T(\span\{y\})^-\!=\H,
$$
where the upper bar $^-$ stands for closure in the norm topology$.$ Thus a
non\-zero ${y\in\H}$ is a supercyclic vector for $T$ if and only if for
every ${x\in\H}$ there exists a $\CC$-valued sequence $\{\alpha_i\}_{i\ge0}$
(which depends on $x$ and $y$ and consists of non\-zero numbers)
such that
$$
\alpha_iT^{n_i}y\conv x
$$
for some subsequence $\{T^{n_i}\}_{i\ge0}$ of $\{T^n\}_{n\ge0}.$ If
${T\in\BH}$ has a supercyclic vector, then it is a {\it supercyclic
operator}$.$ The weak counterpart of the above convergence criterion reads as
follows$.$ A non\-zero vector ${y\in\H}$ is a {\it weakly l-sequentially
supercyclic vector}\/ for an operator ${T\in\BH}$ if for every ${x\in\H}$
there exists a $\CC$-valued sequence $\{\alpha_i\}_{i\ge0}$ (which depends on
$x$ and $y$ and consists of non\-zero numbers) such that, for some
subsequence $\{T^{n_i}\}_{i\ge0}$ of $\{T^n\}_{n\ge0}$,
$$
\alpha_iT^{n_i}y\wconv x.
$$
An operator $T$ in $\BH$ is {\it weakly l-sequentially supercyclic}\/ if it
has a weakly l-se\-quentially supercyclic vector$.$ An equivalent definition
reads as follows$.$ The weak limit set of $\O_T(\span\{y\})$ is the set
consisting of all weak limits of weakly convergent $\O_T(\span\{y\})$-valued
sequences, and a operator ${T\in\BH}$ is weakly l-sequentially supercyclic if
there exists a vector ${y\in\H}$ (called a weakly l-sequentially supercyclic
vector for $T$) for which the weak limit set of $\O_T(\span\{y\})$ is equal
to $\H$.

\vskip4pt
Several forms of weak supercyclicity, including weak l-sequential
supercyclicity, have recently been investigated in
\cite{San1, San2, BM1, MS, Shk, Kub, KD1, KD2}$.$ An operator ${T\in\BH}$
is {\it weakly supercyclic}\/ if there exists a vector ${y\in\H}$ (called a
{\it weakly supercyclic vector}\/ for $T$) such that the projective orbit
$\O_T(\span\{y\})$ is weakly dense in $\H$ (i.e., dense in the weak topology
of $\H$)$.$ A set ${A\sse\H}$ is weakly sequentially closed if every
$A$-valued weakly convergent sequence has its limit in it; and the weak
sequential closure of $A$ is the smallest weakly sequentially closed set
(i.e., the intersection of all weakly sequentially closed sets) including
$A.$ An operator ${T\in\BH}$ is {\it weakly sequentially supercyclic}\/ if
there exists a vector ${y\in\H}$ (called a {\it weakly sequentially
supercyclic vector}\/ for $T$) for which the weak sequential closure of
$\O_T(\span\{y\})$ is equal to $\H.$ Observe that
\vskip6pt\noi
$$
\newmatrix{\!
_{_{\hbox{\ehsc supercyclicity}}} & _{_{_{\textstyle\limply}}}\! &
_{\hbox{\ehsc weak l-sequential}} & _{_{_{\textstyle\limply}}}\! &
_{\hbox{\ehsc weak sequential}}   & _{_{_{\textstyle\limply}}}\! &
_{\hbox{\ehsc weak}}              & \!\!,                           \cr
                                  &                              &
\hbox{\ehsc supercyclicity}       &                              &
\hbox{\ehsc supercyclicity}       &                              &
\hbox{\ehsc supercyclicity}       &                                \cr}
$$
\vskip4pt\noi
and the reverse implications fail (see, e.g., 
\cite[pp.38,39]{Shk}, \cite[pp.259,260]{BM2}).

\vskip4pt
The notion of weak l-sequential supercyclicity was introduced explictly in
\cite{BCS} and implicitly in \cite{BM1}, and investigated in \cite{Shk} who
introduced a terminology similar to the one adopted here (we use the letter
``l'' for ``limit'' in stead of the numeral ``1'' used in \cite{Shk})$.$
Supercyclicity for an operator $T$ implies weak supercyclicity, which in turn
implies the operator $T$ acts on a separable space (see, e.g.,
\cite[Section 3]{KD1}), and so separability for $\H$ is a consequence of any
form of supercyclicity for $T$, including weak l-sequential supercyclicity.

%%%%%%%%%%%%%%%%%%%%%%%%%%%%%%%%%%%%%%%%%%%%%%%%%%%%%%%%%% SECTION 3
\vskip0pt\noi
\section{An Auxiliary Result}

Let $\!\{n\}_{n\ge0}\!$ denote the self-indexing of the set of all
non\-negative integers $\NN_0$ e\-quipped with the natural order$.$ A
subsequence (in fact, a subset) $\{n_k\}_{k\ge0}$ of $\{n\}_{n\ge0}$ is of
{\it bounded increments}\/ (or has {\it bounded gaps}\/) if
${\sup_{k\ge0}(n_{k+1}-n_k)<\infty}.$ (Integer sequences of bounded increments
have been used in \cite{KV} towards weak stability$.$) We say a subsequence
$\{A_{n_k}\}$ of any sequence $\{A_n\}$ is {\it boundedly spaced}\/ if it is
indexed by a subsequence of bounded increments (i.e., $\{A_{n_k}\}$ is
boundedly spaced if $\sup_k{(n_{k+1}-n_k})<\infty)$.

\vskip4pt
A vector ${y\in\H}$ is a {\it collapsing vector}\/ for an operator ${T\in\BH}$
if the orbit of $y$ under $T$ meets the origin (i.e, if ${T^ny=0}$ for some
${n\ge1}).$ Otherwise we say $y$ is a {\it non\-collapsing vector}\/ for $T$
(i.e., if ${T^ny\ne0}$ for every ${n\ge0}).$ If there exists a collapsing
vector for an operator, then we say the operator {\it has a collapsing
orbit}\/$.$ If there exists a non\-collapsing vector for an operator, then we
say the operator has a {\it non\-collapsing orbit}\/$.$ It is clear that a
collapsing orbit is a finite set, and so a weakly supercyclic operator has a
non\-collapsing orbit$.$ Moreover, $T$ has a non\-collapsing orbit if and only
if its adjoint $T^*\!$ has$.$ (Indeed, every vector is collapsing for $T$ if
and only if every vector is collapsing for $T^*$; that is, for every
${y\in\H}$ there exists an ${n\ge1}$ such that ${T^ny=0}$ if and only if
${\<T^ny\,;z\>=0}$ for every ${y,z\in\H}$ for some ${n\ge1}$, which means
${\<y\,;T^{*n}z\>=0}$ for every ${y,z\in\H}$ for some ${n\ge1}$, which is
equivalent to saying $T^{*n}z=0$ for every ${z\in\H}$ for some ${n\ge1}$.)

%%%%%%%%%%%%%%%%%%%%%%%%%%% LEMMA 3.1
\begin{lemma}
Take an arbitrary operator\/ ${T\in\BH}$ and an arbitrary non\-zero vector\/
${x\in\H}.$ Let\/ $\{T^{n_k}\}$ be any subsequence of the power sequence\/
$\{T^n\}$ and let\/ $P_i$ and\/ $P^{\,\prime}_i$ for\/ ${i=1,2}$ stand for
the following properties.
\vskip4pt
\begin{description}
\item{$\kern-7ptP_1\!:\;$}
${T^nx\wconv0}$,
\vskip0pt
\item{$\kern-7ptP^{\,\prime}_1\!:\,$}
${T^{n_k}x\wconv0}$,
\vskip0pt
\item{$\kern-7ptP_2\!:\;$}
${\liminf_n|\<T^nx\,;z\>|>0}$ for some\/ ${z\in\H}$,
\vskip0pt
\item{$\kern-7ptP^{\,\prime}_2\!:\,$}
${\liminf_k|\<T^{n_k}x\,;z\>|>0}$ for some\/ ${z\in\H}$.
\end{description}
\vskip2pt
Moreover, if\/ $T$ has a non\-collapsing orbit, then consider the following
additional properties which hold whenever\/ $T$ is a power bounded operator.
\begin{description}
\vskip6pt
\item{$\kern-7ptP_3\!:\;$}
${\limsup_n|\<T^nx\,;z\>|<\gamma\,\|x\|\,\|z\|}$ for some\/ ${\gamma>0}$
for every non\-collapsing vector\/ $z$ for\/ $T^*\!$,
\vskip0pt
\item{$\kern-7ptP^{\,\prime}_3\!:\,$}
${\limsup_k|\<T^{n_k}x\,;z\>|<\gamma'\|x\|\,\|z\|}$ for some\/ ${\gamma'>0}$
for every non\-collapsing vector\/ $z$ for\/ $T^*\!$,
\end{description}
$($where\/ $\gamma$ and\/ $\gamma'$ may depend on\/ $x)$ and\/
$\gamma'=\gamma$ if\/ $T$ is a contraction.
\vskip4pt\noi
{\rm Claim:} For each\/ ${i=1,2,3}$ the assertions bellow are pairwise
equivalent.
\vskip4pt
\begin{description}
\item{$\kern-6pt$\rm(a)$\kern1pt$}
$P_i$ holds.
\vskip4pt
\item{$\kern-6pt$\rm(b)$\kern1pt$}
$P^{\,\prime}_i$ holds for some boundedly spaced subsequence\/
$\{T^{n_k}\}$ of\/ $\{T^n\}$.
\vskip4pt
\item{$\kern-6pt$\rm(c)$\kern2pt$}
$P^{\,\prime}_i$ holds for every boundedly spaced subsequence\/
$\{T^{n_k}\}$ of\/ $\{T^n\}$.
\end{description}
\end{lemma}

\begin{proof}
We split the proof into 2 parts.

%%%%%%%%%%%%%%%%%%%%%%%%%%% PART 1
\vskip6pt\noi
{\bf Part 1.}
{\it Let\/ $\{\alpha_{n_k}\}$ be a boundedly spaced}\/ (infinite)
{\it subsequence of an}\/ (infinite) {\it se\-quence of complex numbers\/
$\{\alpha_n\}$, let\/ $\alpha$ be an arbitrary complex number, and let
$\beta$ and $\gamma$ be arbitrary non\-negative and positive real numbers,
respectively}\/.
\vskip4pt
\begin{description}
\item{$\kern-12pt({\hbox{a$_1$}})\kern0pt$}
{\it ${\alpha_{n_k+j}\to\alpha}$ as\/ ${k\to\infty}$ for every\/ ${j\ge0}$
if and only if\/ ${\alpha_n\to\alpha}$ as}\/ ${n\to\infty}$.
\vskip4pt
\item{$\kern-12pt({\hbox{b$_1$}})\kern0pt$}
{\it ${\liminf_k|\alpha_{n_k+j}\!-\alpha|>\beta}$ for every\/ ${j\ge0}$ if
and only if}\/ ${\liminf_n|\alpha_n-\alpha|>\beta}$.
\vskip4pt
\item{$\kern-12pt({\hbox{c$_1$}})\kern1pt$}
{\it ${\limsup_k|\alpha_{n_k+j}\!-\alpha|<\gamma}$ for every\/ ${j\ge0}$
if and only if}\/ ${\limsup_n|\alpha_n-\alpha|<\gamma}$.
\end{description}

\vskip6pt\noi
{\it Proof of Part}\/ 1.
This is an easy consequence of the elementary fact that if
$\{n_k\}_{k\ge0}$ is of bounded increments with $M=\sup_k({n_{k+1}-n_k})$,
then ${\bigcup_{j\in[0,M]}\{n_k+j\}_{k\ge0}}=\{n\}_{n\ge0}=\NN_0.\!\!\!\qed$

%%%%%%%%%%%%%%%%%%%%%%%%%%% PART 2
\vskip6pt\noi
{\bf Part 2.}
{\it Take any\/ ${x\in\H}$ and let\/ $\{T^{n_k}\}$ be a boundedly spaced
subsequence of}\/ $\{T^n\}$.
\vskip4pt
\begin{description}
\item{$\kern-11pt({\hbox{a$_2$}})\kern2pt$}
{\it ${\<T^{n_k}x\,;z\>\to0}$ for every\/ ${z\in\H}$ if and only if\/
${\<T^nx\,;z\>\to0}$ for every\/ ${z\in\H}$
\vskip2pt\noi
{\rm(i.e.,} ${T^{n_k}x\wconv0}$ if and only if}\/ ${T^nx\wconv0})$.
\vskip4pt
\item{$\kern-11pt({\hbox{b$_2$}})\kern2pt$}
{\it ${\liminf_k|\<T^{n_k}x\,;z\>|>0}$ for some\/ ${z\in\H}$
if and only if\/ ${\liminf_n|\<T^nx\,;z\>|>0}$ for some}\/ ${z\in\H}$.
\end{description}
\vskip2pt\noi
{\it If\/ $T$ is power bounded and has a non\-collapsing orbit, then}
\vskip2pt\noi
\begin{description}
\item{$\kern-11pt({\hbox{c$_2$}})\kern2pt$}
{\it ${\limsup_k|\<T^{n_k}x\,;z\>|}\kern-1pt<\kern-1pt{\gamma\|x\|\|z\|}$
for every non\-collapsing vector $z$ for\/ $T^*\!$ if and ${\kern-1pt}$only
if\/ ${\limsup_n|\<T^nx\,;z\>|}\kern-1pt<\kern-1pt{\gamma'\|x\|\|z\|}$
for every non\-collapsing vector $z$ for}\/ $T^*\!$,
\end{description}
{\it for some\/ ${\gamma,\gamma'>0}$ where\/ $\gamma'=\gamma$ if\/ $T$ is
a contraction}\/.

\vskip6pt\noi
{\it Proof of Part}\/ 2.
Let $\{T^{n_k}\}$ be a boundedly spaced subsequence of $\{T^n\}$ so that
$\{\<T^{n_k}x\,;z\>\}$ is a boundedly spaced subsequence of
$\{\<T^nx\,;z\>\}$ for ${x,z\!\in\!\H}.$ Recall$:$ $T^{m+n}={T^mT^n}$
for every ${m,n\ge0}.$ Take ${x,y\in\H}$ arbitrary so that for each
${j\ge0}$
$$
\<T^{n_k+j}x\,;y\>=\<T^{n_k}x\,;T^{*j}y\>.
$$
Let $x$ be an arbitrary vector in $\H$.

\vskip4pt\noi
(a$_2$) Suppose ${\<T^{n_k}x\,;z\>\to0}$ for every ${z\in\H}.$ In particular,
${\<T^{n_k}x\,;z\>\to0}$ for every ${z\in\O_{T^*}(y)}$ for every ${y\in\H}.$
This means ${\<T^{n_k}x\,;T^{*j}y\>\to0}$ for every ${j\ge0}$ and every
${y\in\H}$; that is, ${\<T^{n_k+j}x\,;y\>\to0}$ for every ${j\ge0}$ and every
${y\in\H}.$ This is equivalent to ${\<T^nx\,;y\>\to0}$ for every
${y\in\H}$ by (a$_1$) in Part 1$.$ The con\-verse is trivial.

\vskip6pt\noi
(b$_2$)
We prove (b$_2$) contrapositively$.$ Suppose ${\liminf_k|\<T^{n_k}x\,;z\>|=0}$
for every $z$ in $\H.$ Then ${\liminf_k|\<T^{n_k}x\,;z\>|=0}$ for every $z$ in
$\O_{T^*}(y)$ and every $y$ in $\H$, which means
${\liminf_k|\<T^{n_k+j}x\,;y\>|=0}$ for every ${j\ge0}$ and every ${y\in\H}.$
By (b$_1$) in Part 1 this is equivalent to ${\liminf_n|\<T^nx\,;y\>|=0}$ for
every ${y\in\H}.$ The converse is trivial$.$

\vskip6pt\noi
(c$_2$) Let ${T\in\BH}$ be a power bounded operator with a non\-collapsing
orbit (and so its adjoint has a non\-collapsing orbit as well)$.$ Take an
arbitrary positive number $\gamma.$ Suppose
${\limsup_k|\<T^{n_k}x\,;y\>|<\gamma\|x\|\|y\|}$ for every non\-collapsing
vector ${y\in\H}$ for $T^*$ (i.e., for every vector $y$ such that
${T^{*j}y\ne0}$ for every ${j\ge0}$ --- such vectors do exist since $T^*$ has
a non\-collapsing orbit)$.$ Take an arbitrary non\-collapsing vector $y$ for
$T^*$ and an arbitrary ${z\in\O_{T^*}(y)}$ so that ${0\ne z=T^{*j}y}$ for
some ${j\ge0}$ and hence $T^{*i}z=T^{*(j+i)}y\ne0$ for every ${i\ge0}$,
which means $z$ is a non\-collapsing vector for $T^*$ as well$.$ Thus
${\limsup_k|\<T^{n_k}x\,;z\>|}<{\gamma\|x\|\|z\|}$ for every
${z\in\O_{T^*}(y)}$ and every non\-collapsing vector $y$ for $T^*\!.$ This
implies
${\limsup_k|\<T^{n_k+j}x\,;y\>|}={\limsup_k|\<T^{n_k}x\,;T^{*j}y\>|}
<{\gamma\|x\|\|T^{*j}y\|}\le{\gamma'\|x\|\|y\|}$
for every ${j\ge0}$ whenever $y$ is an arbitrary non\-collapsing vector for
$T^*$ (since ${\|T^{*j}y\|\ne0}$ every ${j\ge0}$) with
${\gamma'\ge{\gamma\,\sup_n\|T^n\|}}$
(since $T$ is power bounded --- if $T$ is a contraction so that
${\sup_n\|T^n\|\le1}$ then we may take ${\gamma'=\gamma}).$ Therefore
${\limsup_n|\<T^nx\,;y\>|}<{\gamma'\|x\|\|y\|}$ for every non\-collapsing
vector $y$ for $T^*$ by (c$_1$) in Part 1$.$ Again the converse is
trivial$.\!\!\!\qed$

\vskip4pt\noi
Thus the claimed equivalences in (a), (b) and (c) follow from Part 2 since
$\{T^{n_k}\}$ was taken to be an arbitrary boundedly spaced subsequence of
$\{T^n\}$.
\end{proof}

\vskip4pt
Lemma 3.1 is naturally extended to normed spaces by replacing inner product
with dual pairs, and Hilbert-space adjoints with normed-space adjoints.

%%%%%%%%%%%%%%%%%%%%%%%%%%%%%%%%%%%%%%%%%%%%%%%%%%%%%%%%%% SECTION 4
\vskip0pt\noi
\section{Weak Supercyclicity and Weak stability}

It was shown in \cite[Theorem 2.2]{AB} that {\it if a power bounded operator
on a \hbox{Banach} space is supercyclic, then it is strongly stable}\/$.$
Such a result naturally prompts the question$:$ {\it does weak supercyclicity
imply weak stability for power bounded operators}\/$?$ The question was
posed and investigated in \cite{KD1}, and remains unanswered even if
\hbox{Banach}-space power bounded operators are restricted to
\hbox{Hilbert}-space contractions, where the problem is equivalently stated
for unitary operators (see Remark 4.1 below), and also if weak supercyclicity
is strengthened to weak l-sequential supercyclicity$:$ {\it does there exist
a weakly unstable and weakly l-sequentially supercyclic unitary operator}\/?

\vskip6pt
No unitary operator is supercyclic (reason: no Banach-space isometry is
super\-cyclic \cite[Proof of Theorem 2.1]{AB}) and, in addition to this, no
Hilbert-space hyponormal operator is supercyclic \cite[Theorem 3.1]{Bou}$.$
But the existence of weakly supercyclic (weakly l-sequentially
supercyclic, actually) unitary operators was shown in
\cite[Example 3.6, pp.10,12]{BM1} (also see \cite[Question 1]{Shk}), and the
existence of weakly supercyclic unitary operators that are not
l-sequentially supercyclic was shown in \cite[Proposition 1.1 and
Theorem 1.2]{Shk}$.$ Next we consider in Theorem 4.1 the case of weakly
l-sequentially supercyclic unitary operators $U\!$ that are not weakly
stable in light of boundedly spaced subsequences of the power sequence
$\{U^n\}$, whose proof applies \cite[Theorem 6.2]{KD1} and Lemma 3.1.

%%%%%%%%%%%%%%%%%%%%%%%%%%% THEOREM 4.1
\begin{theorem}
If a unitary operator\/ $U\!$ on a Hilbert space\/ $\H$ is weakly
l-sequentially supercyclic but not weakly stable, then there exists a weakly
l-sequentially supercyclic vector\/ ${y\in\H}$ for\/ $U\!$ such that\/
${U^ny\notwconv0}.$ Moreover, for every weakly l-sequentially supercyclic
vector\/ ${y\in\H}$ for\/ $U\!$ either
\vskip4pt
\begin{description}
\item{$\kern-6pt$\rm(a)$\kern2pt$}
$\liminf_n|\<U^{n_k}y\,;z\>|=0$
for every\/ $($equivalently, for some\/$)$ boundedly spaced subsequence
$\{U^{n_k}\}$ of $\{U^n\}$, $\;\;$ or
\vskip4pt
\item{$\kern-6pt$\rm(b)$\kern2pt$}
$\limsup_k|\<U^{n_k}y\,;z\>|<\|z\|\kern1pt\|y\|$ for some subsequence\/
$\{U^{n_k}\}$ of\/ $\{U^n\}$.
\end{description}
\vskip4pt\noi
Furthermore, if\/ {\rm(a)} fails and if the subsequence in\/ {\rm(b)} is
boundedly spaced, then
$$
0<\liminf_{\phantom|_{\scriptstyle n}}|\<U^ny\,;z\>|
<\limsup_n|\<U^ny\,;z\>|<\|z\|\kern1pt\|y\|.
$$
\end{theorem}

\begin{proof}
We begin with a definition$.$ A normed space $\X$ is said to be of
{\it type 1}\/ if strong convergence (i.e., convergence in the norm
topology) for an arbitrary $\X$-valued sequence $\{x_k\}$ coincides with weak
convergence plus convergence of the norm sequence $\{\|x_k\|\}$ (i.e.,
${x_k\conv x}$ $\iff$ $\big\{{x_k\wconv x}$ and ${\|x_k\|\to\|x\|}\big\}$
--- also called {\it Radon--Riesz space}\/ and the {\it Radon--Riesz
property}\/, respectively)$.$ Every Hilbert space is a Banach space of type 1.

%%%%%%%%%%%%%%%%%%%%%%%%%%% PART 1
\vskip6pt\noi
{\bf Part 1.}
{\it
If a power bounded operator\/ $T$ on a type\/ $1$ normed space\/ $\X$ is
weakly l-sequentially supercyclic, then either
\vskip2pt
\begin{description}
\item{$\kern-4pt$\rm(i)$\kern2pt$}
$T$ is weakly stable,$\;\;$ or
\vskip2pt
\item{$\kern-6pt$\rm(ii)$\kern2pt$}
the set
\vskip-1.5pt\noi
$$
M_T=
\big\{y\in\X\!:\,
\hbox{\rm $y$ is a weakly l-sequentially supercyclic}
$$
\vskip-2.5pt\noi
$$
\kern102pt
\hbox{\rm vector for\/ $T\kern-1pt$ such that\/ ${T^ny\notwconv0}$}\/\big\}
$$
\vskip1pt\noi
is non\-empty, and if\/ $y$ is any vector in\/ $M_T$, then for every nonzero\/
${f\in\X^*}$ such that\/ ${f(T^ny)\not\to0}$ either
\vskip2pt
\begin{description}
\item{$\kern8pt$\rm(a$'$)$\kern2pt$}
$\liminf_n|f(T^ny)|=0,\;\;$ or
\vskip2pt
\item{$\kern8pt$\rm(b$'$)$\kern2pt$}
$\limsup_k|f(T^{n_k}y)|<\|f\|\kern-1pt\limsup_k\|T^{n_k}y\|$ for some
subsequence
\vskip0pt
$\kern14pt\{T^{n_k}\}$ of\/ $\{T^n\}$.
\end{description}
\end{description}
}
\vskip2pt\noi
(Where $\X^*$ stands for the dual of $\X.$) This is Theorem 6.2 from
\cite{KD1}$.$ If a vector ${y\in\X}$ is such that ${T^ny\wconv0}$, then
(a$'$) is tautolo\-gically satisfied for every ${f\in\X^*}\!$, and hence
alternative (ii) can be rewritten as
\vskip2pt
\begin{description}
\item{$\kern-6pt$\rm(ii)$\kern2pt$}
{\it if\/ ${y\in\X}$ is any weakly l-sequentially supercyclic vector for\/
$T\kern-1pt$, then for an arbitrary\/ ${f\in\X^*}$ either\/ {\rm(a$'$)} or\/
{\rm(b$'$)} holds true}\/.
\end{description}
\vskip2pt\noi
If ${\X=\H}$ is a Hilbert space and if $T$ is a contraction, then
$T={C\oplus U\!}$ is uniquely a direct sum of a completely nonunitary
contraction $C$ and a unitary operator $U\!$ (where any of the these direct
summands may be missing) by Nagy--Foia\c s--Langer decomposition for
Hilbert-space contractions (see, e.g., \cite[p.8]{NF} or \cite[p.76]{MDOT})$.$
Since every completely nonunitary contraction is weakly stable (see, e.g.,
\cite[p.55]{Fil} or \cite[p.106]{MDOT}), the above result when restricted to
contractions is equivalent to the case of plain unitary operators, and this
is stated as follows$.$ {\it If a unitary operator\/ $U\!$ on a Hilbert space
$\H$ is weakly l-sequentially supercyclic, then either
\vskip2pt
\begin{description}
\item{$\kern-4pt$\rm(i)$\kern2pt$}
$U\!$ is weakly stable,$\;\;$ or
\vskip2pt
\item{$\kern-6pt$\rm(ii)$\kern2pt$}
$\kern-2pt$if\/ ${y\in\H}$ is a weakly l-sequentially supercyclic vector
for\/ $U\!$, then for an arbitrary\/ ${z\in\H}$ either
\vskip2pt
\begin{description}
\item{$\kern8pt$\rm(a$''$)$\kern2pt$}
$\liminf_n|\<U^ny\,;z\>|=0,\;\;$ or
\vskip2pt
\item{$\kern10pt$\rm(b)$\kern5pt$}
$\limsup_k|\<U^{n_k}y\,;z\>|<\|z\|\kern1pt\|y\|$ for some subsequence\/
$\{U^{n_k}\}$ of\/ $\{U^n\}$.
\end{description}
\end{description}
}
\vskip2pt\noi
(Because the completely nonunitary part of any contraction is necessarily
weakly stable; and unitary operators are isometries.)

%%%%%%%%%%%%%%%%%%%%%%%%%%% PART 2
\vskip6pt\noi
{\bf Part 2.}
Suppose a unitary operator $U\!$ is weakly l-sequentially supercyclic and not
weakly stable$.$ Let $y$ be a weakly l-sequentially supercyclic vector for
$U\!.$ Thus either (a$''$) or (b) holds$.$ But property (a$''$) holds if
and only if (for an arbitrary ${z\in\H}$)
\vskip2pt
\begin{description}
\item{}
\vskip2pt
\begin{description}
\item{$\kern21pt$\rm(a)$\kern2pt$}
$\liminf_k|\<U^{n_k}y\,;z\>|=0$ {\it for every\/ $($equivalently, for
some\/$)$
\vskip0pt
\hskip15pt
boundedly spaced subsequence}\/ $\{U^{n_k}\}$ of $\{U^n\}.$
\end{description}
\end{description}
\vskip2pt\noi
Indeed, by Lemma 3.1($P_2,P^{\,\prime}_2$) $\liminf_n|\<U^ny\,;z\>|=0$ for
every ${z\in\H}$ if and only if $\liminf_k|\<U^{n_k}y\,;z\>|=0$ for every
${z\in\H}$ for every (equivalently, for some) boundedly spaced subsequence
$\{U^{n_k}\}$ of $\{U^n\}.$ Hence (a$''$) holds if and only if (a) holds$.$
On the other hand, if (a$''$) fails, then (b) holds$.$ Thus suppose (a$''$)
fails; equivalently, suppose (a) fails$.$ Fix an arbitrary weakly
l-sequentially supercyclic vector ${y\in\H}$ for $U$ for which (a$''$)
fails so that (b) holds$.$ Since $U$ is a weakly l-sequentially super\-cyclic
operator, it has a non\-collapsing vector, and so has $U^*\!.$ Let ${z\in\H}$
be an arbitrary non\-collapsing vector for $U^*\!.$ Now suppose (b) holds for
some {\it boundedly spaced}\/ sub\-sequence$.$ Then
Lemma 3.1($P_3,P^{\,\prime}_3$) (with ${\gamma=\gamma'=1}$) ensures
\goodbreak
\vskip4pt\noi
$$
0<\liminf_{\phantom|_{\scriptstyle n}}|\<U^ny\,;z\>|
\le\limsup_n|\<U^ny\,;z\>|<\|z\|\kern1pt\|y\|
$$
\vskip-1pt\noi
(since (a$''$) fails)$.$ If
\vskip-1pt\noi
$$
\liminf_{\phantom|_{\scriptstyle n}}|\<U^ny\,;z\>|=\limsup_n|\<U^ny\,;z\>|,
$$
then $\{|\<U^ny\,;z\>|\}$ converges to a positive number ${\alpha(y,z)}$,
$$
|\<U^ny\,;z\>|\to\alpha(y,z).
$$
Since $y$ is an l-sequentially weakly supercyclic vector for $U\!$, for
each ${x\in\H}$ there exists a scalar sequence $\{\alpha_i(y,x)\}$ such
that ${\alpha_i(y,x)\,U^{n_i}y\wconv x}$ for some subsequence
$\{U^{n_i}\}_{i\ge0}$ of $\{U^n\}_{n\ge0}$, which implies
$$
|\alpha_i(y,x)|\,|\<U^{n_i}y\,;w\>|\to|\<x\,;w\>|
$$
for every ${w\in\H}.$ Hence, since ${|\<U^ny\,;z\>|\to\alpha(y,z)\ne0}$,
$$
|\alpha_i(y,x)|\to\smallfrac{|\<x\,;z\>|}{\alpha(y,z)}
$$
for every ${x\in\H}$, which is a contradiction because $\alpha_i(y,x)$ does
not depend on $z$ and $\alpha(y,z)$ does not depend on $x.$ (Indeed, since
$z$ was taken to be an arbitrary non\-collapsing vector for $U\!$, the above
limit holds for every such a $z$ and every ${x\in\H}$, and so take ${x\in\H}$
and a pair of non\-collapsing vectors ${z,z'}$ for $U^*$ such that $x$ is
orthogonal to $z$ but not to $z'.)$ Outcome$:$
$$
\liminf_{\phantom|_{\scriptstyle n}}|\<U^ny\,;z\>|<\limsup_n|\<U^ny\,;z\>|.
$$
\vskip-2pt\noi
Therefore, if the subsequence in (b) is boundedly spaced,
then
$$
0<\liminf_{\phantom|_{\scriptstyle n}}|\<U^ny\,;z\>
<\limsup_n|\<U^ny\,;z\>|<\|z\|\kern1pt\|y\|.
$$
\vglue-23pt
\end{proof}
\vglue2pt

%%%%%%%%%%%%%%%%%%%%%%%%%%% REMARK 4.1
\vskip0pt
\begin{remark}
As we saw in the proof of Theorem 4.1, the Nagy--Foia\c s--Langer
de\-composition for contractions on a \hbox{Hilbert} space $\H$ says that
$\H$ admits an orthogonal decomposition $\H={\U^\perp\!\oplus\U}$, where a
contraction $T={C\oplus U\!}$ is uniquely a direct sum of a completely
non\-unitary contraction $C=$ ${T|_{\U^\perp}\!\in\B[\U]}$ and a unitary
operator $U\!={T|_\U\in\B[\U]}$, where $C$ is the completely non\-unitary part
of $T$ and $U\!$ is the unitary part of $T$ (any of the these parcels may be
missing)$.$ Moreover, a completely non\-unitary contraction is weakly stable
(see, e.g., \cite[pp.76,106]{MDOT})$.$ Thus a \hbox{Hilbert}-space contraction
$T$ is not weakly stable if and only if it has a (nontrivial) unitary part
$U\!$ which is not weakly stable (and so ${\|T\|=\|U\|=1}).$ Therefore the
question ``{\it does weak l-sequential supercyclicity for a Hilbert-space
contraction implies weak stability}\/$?$'' is equivalent to the question
``{\it does weak l-sequential supercyclicity for a unitary operator implies
weak stability}\/?''
\end{remark}

%%%%%%%%%%%%%%%%%%%%%%%%%%%%%%%%%%%%%%%%%%%%%%%% REFERENCES
\vskip-10pt\noi
\bibliographystyle{amsplain}

\end{document}